\documentclass[12pt]{article}

\usepackage[dvips, outer=1in, top=1.2in, bottom=1.2in]{geometry}

\usepackage{graphicx}

\input{mathdefs.sty}

\usepackage[font=small,format=plain,labelfont=bf,up,up]{caption}

\newtheorem*{theorem}{Theorem}

\newtheorem{prop}{Proposition}

\usepackage{fancyhdr}
\usepackage{hyperref}

\begin{document}

\title{Separating Solution of a Quadratic Recurrent Equation}

\author{Ya.\, G.\, Sinai\footnote{Mathematics
Department, Princeton University and Landau Institute of Theoretical Physics, Russian Academy of Sciences} \and I.\, Vinogradov\footnote{Mathematics
Department, Princeton University}}

\date{\today}
\maketitle

\bigskip


\hfill\hbox to 7cm{\vbox{\noindent To J.\, Froehlich and T. Spencer \\with love and admiration.}}

\bigskip


\begin{abstract}
In this paper we consider the recurrent equation $$\Lambda_{p+1}=\frac1p\sum_{q=1}^pf\bigg(\frac{q}{p+1}\bigg)\Lambda_{q}\Lambda_{p+1-q}$$ for $p\ge 1$ with $f\in C[0,1]$ and $\Lambda_1=y>0$ given. We give conditions on $f$ that guarantee the existence of $y^{(0)}$ such that the sequence $\Lambda_p$ with $\Lambda_1=y^{(0)}$ tends to a finite positive limit as $p\to \infty$.
\end{abstract}

\section{Introduction}

The following problem arose in the joint papers of the first author and Dong Li (see \cite{LS1} and \cite{LS2}). Let $f$ be a continuous real-valued function on $[0,1]$. Define the sequence $\Lambda_p$ for $p=1,2,\dots$ by 
\beq\label{eq:basic}\Lambda_{p+1}=\frac1p\sum_{q=1}^pf\bigg(\frac{q}{p+1}\bigg)\Lambda_{q}\Lambda_{p+1-q}\eeq
and set $\Lambda_1=y\ge0$. We shall occasionally write $\Lambda_p(y)$ to emphasize the dependence of $\Lambda_p$ on the initial value $y$. It is clear that $\Lambda_p(cy)=c^p\Lambda_p(y).$ Therefore if $\Lambda_p(y)\to\infty$ as $p\to\infty$ and $c>1$, then $\Lambda_p(y')\to\infty$ as $p\to\infty$ where $y'=cy$. On the other hand if $\Lambda_p(y)\to0$ and $0<c<1$, then $\Lambda_p(y')\to0$. Thus there exist $y^+$ and $y^-$ such that $\Lambda_p(y)\to\infty$ for $y\in(y^+,\infty)$ with $y^+$ as small as possible and $\Lambda_p(y)\to0$ for $y\in(0,y^-)$ with $y^-$ as large as possible. It is a natural question whether $y^+=y^-=y^{(0)}$ and whether $\Lambda_p(y^{(0)})\to\const$ as $p\to\infty$. It is easy to see that this constant must be $\bigg(\displaystyle\int_0^1f(x)dx\bigg)^{-1}$, and it is our first assumption that the last integral is positive. It is enough to consider the case $\displaystyle\int_0^1f(x)dx=1$ because if $\tilde f(x)=Kf(x)$ for a constant $K$, then $\tilde\Lambda_p(y)=K^{-1}\Lambda_p(y).$ If the answer to our question is affirmative then $\Lambda_p(y^{(0)})$ is called the separating solution of $\eqref{eq:basic}$.

This problem was considered previously in \cite{L1} and \cite{S1}. The analysis in \cite{L1} covered the case $f(x)=6x^2-10x+4$ needed in \cite{LS1}. The analysis in \cite{S1} was based on a different idea but unfortunately had a number of gaps. This paper is a modified and corrected version of \cite{S1}. 

Before we give the assumptions we impose on $f$, we remark that $f(x)$ and $f(1-x)$ produce identical sequences. Therefore the existence of a separating solution depends only on $f_1(x)=f(x)+f(1-x).$ Of course establishing existence of a solution for $f$ guarantees its existence for $g$ if $g_1=f_1$. Given $f_1(x)$ one can find $f(x)$ so that $f(1)=0$. Thus  we assume that $f(1)=0$ without loss of generality. Now we impose the following conditions on $f$: 
\begin{enumerate}
\item $f\in C^2[0,1],$
\item $f_1$ is positive on $[0,1]\cap\mathbf Q$,
\item all complex $\sigma\ne 1$ satisfying $\displaystyle\int_0^1t^\sigma f_1(t)dt=1$ have the property that $\re \sigma<0$,
\item a numerical condition to be explained later. 
\end{enumerate}
Observe that an assumption similar to 2 is necessary as $\Lambda_p$ will vanish for $p$ sufficiently large if $f_1$ vanishes on too large a set (e.g., if $f_1(\frac12)=0$); Assumption 2 effectively ensures that $\Lambda_p>0$ for all $p$. Finally we introduce functions $f_2(x)=-(x f(x))'$ and $f_3(x)=-\displaystyle\frac1{x^2}\displaystyle\int_0^x tf_2(t)dt$.

Define $a_p>0$ for $p\ge1$ by the condition $\Lambda_p(a_p)=1$; Assumption 2 above makes this possible. The strategy of the proof will be to show that $a_p\to a_\infty$ sufficiently rapidly. Take positive constants $A$ and $B$ with $B<1<A$ and consider the inequalities \begin{align}\label{eq:inductive}B\le  a_p && |a_{p}-a_{p-1}|\le A/p^{2+\delta};\end{align} where $p$ is given and $\delta\in(0,\frac12)$ will be chosen later and will depend on $f_1$. 

\begin{theorem}[Main Theorem]Let $f$ satisfy assumptions 1--3 above. If for some $p_0$ (depending on $A$, $B$, and $f_1$) the inequalities \eqref{eq:inductive} hold for $p\le p_0$, then they are valid for all $p\ge1$. 
\end{theorem}

Our proof will be inductive. We shall assume \eqref{eq:inductive} for $p\le r$ and prove it for $p=r+1$. This will imply that the limit $\displaystyle\lim_{p\to\infty}a_p=a_\infty$ exists and $\Lambda_p(a_\infty)$ will be the desired separating solution. 

The rest of the paper is structured as follows. In Section 2 we derive a recurrent equation for $a_p$. In Section 3 we solve this equation using the inductive hypothesis. The last Section consists  of numerical analysis and outlines further research on the problem. 

The first author thanks NSF for the financial support, grant DMS N 0600996.


\section{Recurrent Equation for $a_p$}

We shall denote absolute constants by $C$ with superscipts in the course of this calculation. We have that  \beq\label{eq:beforecalculation}\Lambda_{p+1}(a_{p+1})-\Lambda_{p+1}(a_p)=-(\Lambda_{p+1}(a_p)-\Lambda_p(a_p)).\eeq

Put $\gamma =
\frac{p_1}{p}$, $p_2 = p - p_1$, $\gamma' =  \frac{p_1}{p + 1}$.
Then
\begin{align*} \Lambda_{p + 1} ( a_p) & =  \frac{1}{p}  \sum\limits_{p_1 =
1}^{p} f ( \gamma')  \Lambda_{p_1} ( a_p)  \Lambda_{p_2 +1}
( a_p )  
 =  \frac{1}{p}   \sum\limits_{p_1 = 1}^{p}  f ( \gamma' )
( \Lambda_{p_1} ( a_p) - 1) ( \Lambda_{p_2+1} ( a_p ) - 1)  + 
\\
& +  \frac{1}{p}   \sum\limits_{p_1 = 1}^{p}  f ( \gamma') (
\Lambda_{p_1} ( a_p ) - 1)  +  \frac{1}{p}  \sum\limits_{p_1 =
1}^{p}  f ( \gamma') ( \Lambda_{p_2+1} (a_p) - 1) 
 -  \frac{1}{p}  \sum\limits_{p_1 = 1}^{p}  f ( \gamma')
 =
\\
& =  \frac{1}{p}  \sum\limits_{p_1 = 1}^{p}  f ( \gamma') (
\Lambda_{p_1} ( a_p) - 1) ( \Lambda_{p_2+1} (a_p) - 1)  +  
\frac{1}{p}  \sum\limits_{p_1 = 1}^{p}  f_1 ( \gamma' ) (
\Lambda_{p_1} ( a_p) - 1) - \frac{1}{p}  \sum\limits_{p_1 = 1}^{p}  f
( \gamma')  .
\end{align*}
A similar formula can be written for $\Lambda_p ( a_p)$:
\begin{align*} \Lambda_p ( a_p ) & =  \frac{1}{p-1}  
\sum\limits_{p_1 = 1}^{p-1}  f ( \gamma)  ( \Lambda_{p_1} ( a_p) - 1)
( \Lambda_{p_2} ( a_p) - 1)  + 
\\
&+  \frac{1}{p-1}  \sum\limits_{p_1 = 1}^{p-1}  f_1 ( \gamma ) (
\Lambda_{p_1} ( a_p) - 1)  -  \frac{1}{p-1}  \sum\limits_{p_1 =
1}^{p-1}  f ( \gamma )  .
\end{align*}
Subtracting $\Lambda_p ( a_p)$ from $\Lambda_{p+1} ( a_p)$ we get 
\begin{align*} \Lambda_{p+1} ( a_p) - \Lambda_p ( a_p ) & =  
\frac{1}{p}  f \left( \frac{p}{p+1} \right) ( \Lambda_p ( a_p ) - 1) (
\Lambda_1 (a_p ) - 1) 
\\
& +  \sum\limits_{p_1 = 1}^{p-1}  \left( \frac{1}{p}  f (
\gamma' ) - \frac{1}{p-1}  f ( \gamma ) \right) ( \Lambda_{p_1} (
a_p ) - 1) ( \Lambda_{p_2} ( a_p ) - 1 ) 
\\
&+  \frac{1}{p}  \sum\limits_{p_1 = 1}^{p-1}  f ( \gamma')
( \Lambda_{p_1} ( a_p) - 1) ( \Lambda_{p_2 + 1} (a_p ) -
\Lambda_{p_2} (a_p)) 
\\
 &+  \frac{1}{p}  f_1 \left( \frac{p}{p+1} \right) ( \Lambda_{p-1}
(a_p) - 1)  +  
\\
&+\sum\limits_{p_1 = 1}^{p-1}  \left( \frac{1}{p}   f_1 (
\gamma' ) -  \frac{1}{p-1}  f_1 ( \gamma ) \right)( \Lambda_{p_1} ( a_p) - 1)  +  \\
&+
\frac{1}{p}  
f \left( 
\frac{p}{p+1} \right ) -  
\sum\limits_{p_1 = 1}^{p-1}  \left(
\frac{1}{p}  f ( \gamma' )  -  \frac{1}{p-1}  f ( \gamma )
\right)  =  
\sum\limits_{j = 1}^{7}  I_p^{(j)}  .
\end{align*}

We estimate $I_p^{(j)}$. It will be shown that $I_p^{(5)}$ is the main term while the others have a smaller order of magnitude. This term produces the recurrent equation that we shall analyze in Section \ref{sec:analysis}

It is readily seen that $I_p^{(4)}=\eps_p^{(1)}$, where $|\eps_p^{(1)}|\le\frac{C^{(1)}A}{Bp^{2+\delta}}.$ The reasoning is as follows. Rewrite the term as $$\frac1p\bigg(f_1(1)-f'(\xi)\frac1{p+1}\bigg)\bigg(\Lambda_{p-1}(a_{p-1})\bigg(\frac{a_p}{a_{p-1}}\bigg)^{p-1}-1\bigg).$$ It is clear how to bound the second term in the first factor. The second factor can be written as $$\sum_{k=1}^{p-1}\bigg(\frac{a_p-a_{p-1}}{a_{p-1}}\bigg)^k\binom{p-1}{k}$$ whence it is easy to see that it is bounded by $\const \cdot\frac{A}{Bp^{1+\delta}}.$ The estimate for the fourth term follows. 

We go on to $$I_p^{(5)}=\sum_{p_1=1}^{p-1}\bigg(\frac1pf_1(\gamma')-\frac1{p-1}f_1(\gamma)\bigg)(\Lambda_{p_1}(a_p)-1).$$ For the first factor in the sum we get $$\frac{f_2(\gamma')}{p(p-1)}+\eps_p^{(2)}$$ where $|\eps_p^{(2)}|\le\frac{C^{(2)}}{p^3}.$ The second factor is more complicated and we first rewrite it as
$$\frac{a_p-a_{p_1}}{a_{p_1}}p_1-\sum_{k=2}^{p_1}\bigg(\frac{a_p-a_{p_1}}{a_{p_1}}\bigg)^k\binom{p_1}{k}.$$ 
The last term of this expression is not more than $\frac{C^{(3)}}{p_1^{2\delta}}\left(\frac AB\right)^2.$ Multiplying out gives the following expression: 
$$I_p^{(5)}=\sum_{p_1=1}^p\frac{\gamma f_2(\gamma')}{p-1}\frac{a_p-a_{p_1}}{a_{p_1}}+\eps_p^{(3)}$$ and $|\eps_p^{(3)}|\le C^{(4)}\frac1{p^{1+2\delta}}\left(\frac AB\right)^2.$ 

Now we deal with three relatively simple terms. Let us begin with the seventh one: 
\begin{align*}
I_p^{(7)} & = - \sum\limits_{p_1 = 1}^{p - 2}  
\left( \frac{1}{p}   f ( \gamma^\prime)   - 
\frac{1}{p -1}   f ( \gamma) \right)   =  
\\
& = -   \sum\limits_{p_1 = 1}^{p - 2}   
\left[
\left(
\frac{1}{p}   f \left( \frac{p_1}{p + 1} \right)   - \frac{1}{p - 1}
f \left( \frac{p_1}{p+1} \right) \right)   +   
\frac{1}{p - 1}   
\left( f \left( \frac{p_1}{p + 1} \right)   -   f 
\left( \frac{p_1}{p} \right) \right) \right] =
\\
& =    \frac{1}{p ( p - 1 )}  
\sum\limits_{p_1 = 1}^{p - 2}   \left[
f   \left( \frac{p_1}{p + 1} \right)   +  
\frac{p_1}{(p + 1 )}   f^\prime \left( \frac{p_1}{p + 1} \right )
\right]   +   \epsilon_{p}^{(4)} =
\\
& = \frac{p + 1}{p ( p-1)} 
\int\limits_{0}^{1}   [ f ( \gamma )   +   \gamma   f^\prime (
\gamma ) ]   d \gamma   +   \epsilon_p^{(5)}.\end{align*}
Our assumption that $f(1)=0$ implies that the last integral vanishes. Thus, $I_p^{(7)}=\eps_p^{(5)}$ and $|\eps_p^{(5)}|\le \frac{C^{(5)}}{p^2}.$ 

It is easy to see that $I_p^{(1)}=0$ and $|I_p^{(6)}|\le \frac{C^{(6)}}{p^2}.$

To estimate $I_p^{(2)}$ we rewrite it as \begin{gather*}I_p^{(2)}=\sum_{p_1=1}^p\left(-\frac{\gamma'f'(\gamma')+f(\gamma')}{p(p-1)}+\eps_p^{(6)}\right)
\left[\sum_{k=1}^{p_1}\left(\frac{a_p-a_{p_1}}{a_{p_1}}\right)^k\binom{p_1}{k}\right]
\left[\sum_{k=1}^{p_2}\left(\frac{a_p-a_{p_2}}{a_{p_2}}\right)^k\binom{p_2}{k}\right]\end{gather*} The terms in the brackets are bounded by $C^{(7)}\frac{A}{Bp_1^\delta}$ and $C^{(8)}\frac{A}{Bp_2^\delta}.$ Thus the estimate for this term becomes $C^{(9)}\frac{A}{Bp^{1+2\delta}}.$ 

Finally for the third term we need to estimate $$\Lambda_{p_2+1}(a_p)-\Lambda_{p_2}(a_p).$$ It is not difficult to see that $|\Lambda_{p_2+1}(a_p)-\Lambda_{p_2}(a_p)|\le\frac{C^{(10)}}{p_2^{1+\delta}}\left(\frac AB\right)^3.$ Combining this with the remaining factors gives the bound $C^{(11)}\frac 1{p^{1+2\delta}}\left(\frac AB\right)^5$ for $I_p^{(3)}$. We have used the fact that $f(x)\le C(1-x)$ in the last step. 

Now we can put the seven terms together and see that 
\beq \notag\label{eq:seventermsdone}\Lambda_{p+1}(a_{p+1})-\Lambda_{p+1}(a_{p})=-\sum_{p_1=1}^p \frac{\gamma f_2(\gamma')}{p-1}\frac{a_p-a_{p_1}}{a_{p_1}}+\eps_p^{(7)}\eeq where $|\eps_p^{(7)}|\le \frac{C^{(12)}}{p^{1+2\delta}} \left( \frac AB\right)^5. $ A simple calculation gives the recurrent equation 
\beq \label{eq:recurrentequation}(p+1)\frac{a_{p+1}-a_p}{a_p}=-\sum_{p_1=1}^p \frac{\gamma f_2(\gamma')}{p-1}\frac{a_p-a_{p_1}}{a_{p_1}}+\eps_p^{(8)}\eeq
with $|\eps_p^{(8)}|\le  \frac{C^{(13)}}{p^{1+2\delta}} \left( \frac AB\right)^5. $


Our objective in this section is to derive a recurrent equation for $b_p=p^{2}\frac{a_p-a_{p-1}}{a_{p-1}}$. Thus we rewrite \eqref{eq:recurrentequation} using $b_p$ rather than $a_p$. We take a positive integer $Q\le p$ and  get $$b_{p+1}=-p\left[\sum_{p_1=1}^Q+\sum_{p_1=Q+1}^p\right]\frac{\gamma f_2(\gamma')}{p-1}\frac{a_p-a_{p_1}}{a_{p_1}}+\eps_p^{(9)}.$$ Now the sum from 1 to $Q$ gives a contribution bounded by $\frac{C^{(14)}}{p^2}\frac AB  Q^{1-\delta}.$ For the sum from $Q+1$ to $p$ we observe that $$\prod_{q=p_1+1}^p\left(1+\frac{b_{q}}{p^{2}}\right)-1=\frac{a_p-a_{p_1}}{a_{p_1}}.$$ The left hand side can be written as $$\sum_{q=p_1+1}^p\frac{b_q}{q^{2}}+\eps_{p_1}^{(10)}$$ with $|\eps_{p_1}^{(10)}|\le \frac{C^{(15)}}{p_1^3}\left(\frac AB\right)^2$ provided $Q$ is chosen sufficiently large and independent of $p$. Using this fact we simplify our equation to 
\beq\label{eq:almostfinalrecurrent}b_{p+1}=-p\sum_{p_1=1}^p\frac{\gamma f_2(\gamma')}{p-1}\sum_{q=p_1+1}^p\frac{b_q}{q^{2}}+\eps_p^{(11)},\eeq with $|\eps_p^{(11)}|\le \frac{C^{(16)}}{ p^{2\delta}}\left(\frac AB\right)^5.$ After changing the order of summation we obtain the equation 
\beq\label{eq:finalrecurrent}b_{p+1}=\frac1{p}\sum_{q=2}^pb_qf_3\left(\frac qp\right)+\eps_p^{(12)}\eeq with $|\eps_p^{(12)}|\le \frac{C^{(17)}}{p^{2\delta}}\left(\frac AB\right)^5.$
This is the equation we set out to solve; it is effectively a linearized version of the original equation. 

\section{Analysis of the Recurrent Equation}

\label{sec:analysis}

It will be more advantageous to have a continuous equation rather than a discrete one. To this effect we need to define $b(x)$ that would agree with $b_p$ when $x=p$. First we observe that \eqref{eq:finalrecurrent} can be written as $$b_{p}=\frac1{p}\sum_{q=2}^pb_qf_3\left(\frac qp\right)+\eps_p^{(13)}$$ with a different constant in the estimate for the error term. Now we can extend $b$ as follows: set $b(x)=b_{\lfloor x\rfloor}$ with $b(x)=0$ on $[0,1)$. Then we have $$\int_0^1 b(py)f_3(y)dy=\frac1p\sum_{q=2}^pb_qf_3\left(\frac{q'}{p}\right)$$ with $q'\in[q,q+1]$. It is easy to see that this sum differs from the one in the recurrent equation by not more than $\frac{C}{p^2}\sum_{q=1}^pb_q.$ The new error term $\eps(x)$ will incorporate this term as well as $\eps^{(13)}$. It is also clear that we need to add lower order corrections to $\eps(x)$ to ensure that $b(x)$ remains constant for non-integral $x$. 

The equation to solve is now \beq b(x)=\int_0^1 b(xy)f_3(y) dy+\eps(x).\label{eq:continuous}\eeq  The error term is $|\eps(x)|\le \frac{C^{(18)}}{x^{2\delta}}$ (we are dropping the dependence on $A$ and $B$ for now). 

\begin{prop}
Let $f_3$ be given as before and let $$\Sigma=\left\{\sigma\in\mathbf C\colon\int_0^1 t^\sigma f_3(t)dt=1\right\}.$$ Then all $b(x)$ satisfying \eqref{eq:continuous} with $\eps(x)$ as above are (possibly infinite) linear combinations of elements of $\bigcup_{\sigma\in\Sigma}\{x^\sigma, x^\sigma\log x, \dots, x^\sigma\log^{k-1}x\}\cup\{b_\eps(x)\}$ where $k=k(\sigma)$ denotes the multiplicity of $\sigma$, and the special solution $b_\eps(x)$ has the property $|b_\eps(x)|\le \frac{C^{(19)}}{x^{2\delta}}.$ 
\end{prop}

\begin{proof}
This proof can be carried out in a simpler way using the Mellin Transform, but we shall stick to the Fourier Transform as it is more common. To this end we set $x=e^\xi$, $y=e^{-\eta}$, $B(\xi)=b(e^\xi)$, $F(\eta)=-f_3(e^{-\eta})e^{-\eta}$, $E(\xi)=\eps(e^\xi)$. We also extend $f_3$ to be zero on $(1,\infty)$. We get $$B(\xi)=\int_{-\infty}^\infty B(\xi-\eta)F(\eta)d\eta+E(\xi).$$ Taking Fourier Transform of this equation yields $$\hat B(\alpha)=\frac{\hat E(\alpha)}{1-\hat F(\alpha)}.$$ Of course we only require that these are equal as distributions. Now $\hat F(-i\alpha)=\displaystyle\int_0^1 t^\alpha f_3(t)dt$, so we need to look where it attains the value one. It is precisely on the set $i\Sigma$. To invert $\hat B$, we shall integrate along a countour that goes around points in $i\Sigma$ (one can easily see them to be isolated) and stays on the real line otherwise. The integral away from the poles will give $b_\eps(x)$ and can be bounded as follows. We know that $|E(\xi)|\le Ce^{-2\delta\xi}$ (hence the Fourier Transform is analytic in a strip centered at the real axis) and that $\frac1{1-\hat F(\alpha)}$ is meromorphic. Thus the decay rate for $\left(\frac{\hat E(\alpha)}{1-\hat F(\alpha)}\right)^\vee(\xi)=b_\eps(e^\xi)$ is the same as that for $E(x).$ Integrals near poles evaluate to residues at those poles, up to constants. For a simple pole at $\alpha'$ the residue is $e^{i\xi \alpha'}$. Residues at higher order poles are obtained in the same way. The result is immediate once we return to the original variables. 
\end{proof}

The next proposition will allow us to better understand the structure of $\Sigma$. 

\begin{prop}
With notation as above, the set $$\Sigma\cap \{\sigma\in \mathbf C\colon\re \sigma > \sigma_0\}$$ is finite for each $\sigma_0>-1$. 
\end{prop}

\begin{proof}
Let us look only at the real part; in this calculation $\sigma=\mu+i\nu$. We have $$\int_0^1 \cos(\nu \log t)t^\mu f_3(t)dt=-\frac1{\nu}\int_0^1 \sin(\nu\log t)\frac{d}{dt}(f_3(t)t^{\mu+1})dt.$$ It is clear that the last expression tends to zero uniformly in $\mu$ as $\nu\to\infty$ provided $\mu>\sigma_0>-1$. 
\end{proof}

This Proposition allows us to study $\Sigma$ more carefully. Since $$f_3(t)=f_1(t)-\frac1{t^2}\int_0^t xf_1(x)dx\:\:\mbox{and}\:\int_0^1f_1(t)dt=2,$$ we always have $0\in\Sigma$. Set $F_1(\sigma)=\displaystyle\int_0^1 t^\sigma f_1(t)dt$ and $F_3(\sigma)=\displaystyle\int_0^1 t^\sigma f_3(t)dt$. Then $$\frac{\sigma}{\sigma-1}F_1(\sigma)-\frac1{\sigma-1}=F_3(\sigma)$$ for $\sigma\ne 1.$ This means that it suffices to look for solutions to $F_1(\sigma)=1$. It is easy to see that $F_1(\sigma)<1$ when $\re \sigma>1$ even without Assumption 3. It is also clear that $F_3(1)\ne 1$. Therefore Assumption 3 effectively says that there are no solutions to $F_3(\sigma)=1$ in the strip $0\le \re\sigma\le 1$ with the exception of $\sigma=0$. Thus $\sigma=0$ is the solution with the largest real part. However, this solution  is extraneous to our problem because it implies that $\frac{a_p-a_{p-1}}{a_{p-1}}\sim C/p^2$ and thus $$\left(1-\frac{a_\infty-a_p}{a_p}\right)^p\to e^{-C}.$$ This is only possible when $C=0$, so this solution does not work in our situation. To this end we define $\Sigma'=\Sigma\setminus\{0\}.$

Suppose $\Sigma'$ is nonempty and let $\sigma_1=\max\{\re \sigma\colon \sigma\in\Sigma'\}<0$; it exists by Proposition 1. Then choose $\delta$ so that $\sigma_1<-\delta<0$. Then the slowest decaying solution $b_p$ behaves at worst like $p^{\sigma_1}$ and $$\left(1-\frac{a_\infty-a_p}{a_p}\right)^p\to 1. $$ This means that $a_\infty$ is the desired separating solution. If $\Sigma'$ is empty, define $\delta=\frac14$ and $\sigma_1=-\frac12$. Then the same result holds. It is clear that $A$ and $B$ remain bounded in either case. 

\section{Numerical Analysis}

\begin{figure}[p]
\begin{center}

\includegraphics[width=0.75\textwidth]{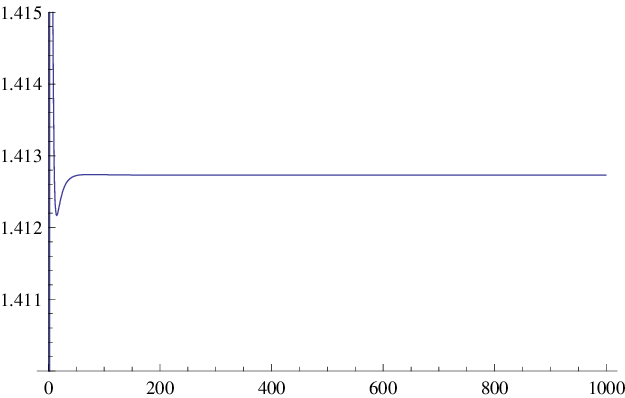}
\caption{The sequence $a_p$ for $f(x)=6x^2-10x+4$.}
\end{center}
\end{figure}

\begin{figure}[p]
\begin{center}

\includegraphics[width=0.75\textwidth]{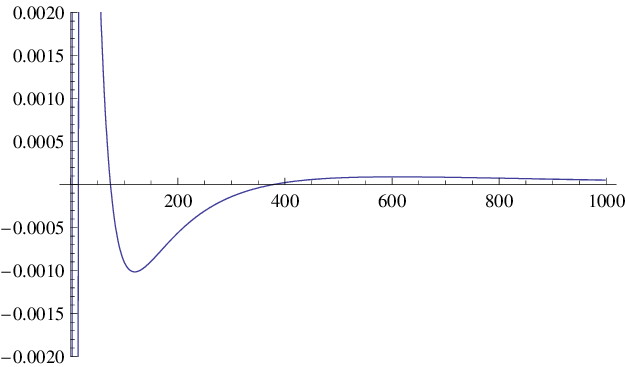}
\caption{The sequence $pb_p$ for $f(x)=6x^2-10x+4$.}
\end{center}
\end{figure}

\begin{figure}[p]
\begin{center}

\includegraphics[width=0.75\textwidth]{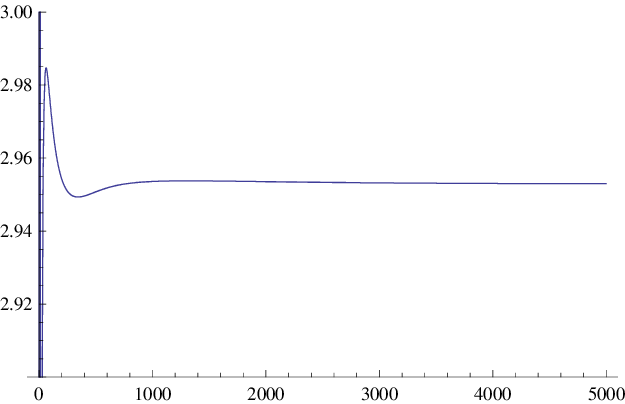}
\caption{The sequence $a_p$ for $f(x)=9x^8$.}
\end{center}
\end{figure}

\begin{figure}[p]
\begin{center}

\includegraphics[width=0.75\textwidth]{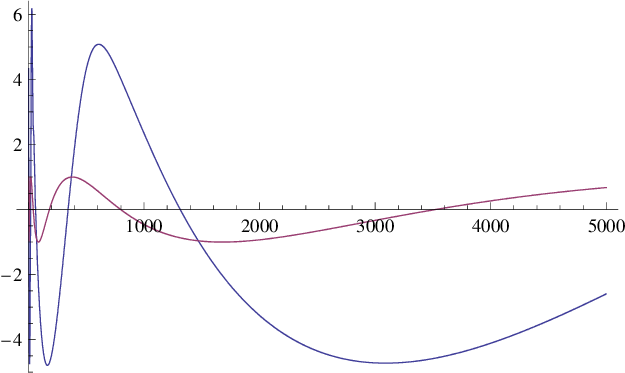}
\caption{The sequences $p^{0.234067}b_p$ and $\cos(2.11581\log p)$ for $f(x)=9x^8$.}
\end{center}
\end{figure}

While the Main Theorem is true for $f$ satisfying Assumption 3, numerical calculations suggest that it is true even without this assumption. First we show how  asymptotics work for a function satisfying this assumption, say $f(x)=6x^2-10x+4$. From Figure 2 we see that $pb_p\to0$, which is expected as $\Sigma'=\varnothing$ for this function. It is reasonable to infer from Figure 1 that a separating solution exists and that the starting value is approximately $1.412729$. 

Next we consider the function $f(x)=9x^8$. It has $\Sigma'\approx\{-0.234067\pm 2.11581i\}$. 
Figure 4 shows $p^{-\sigma_1}b_p$ and $\cos (\im\sigma\log p)$ (the smaller graph is the cosine). We see that the consecutive extrema of the rescaled $b_p$ are at about the same absolute heights. In addition, we note that zeros of the two functions alternate. Therefore, it is plausible that $A\cos(\im\sigma\log p+B)$ will coincide with our function for  sufficiently large $p$. Figure 3 suggests that a separating solution exists and that $a_\infty\approx2.95072$. 

\begin{figure}[p]
\begin{center}

\includegraphics[width=0.75\textwidth]{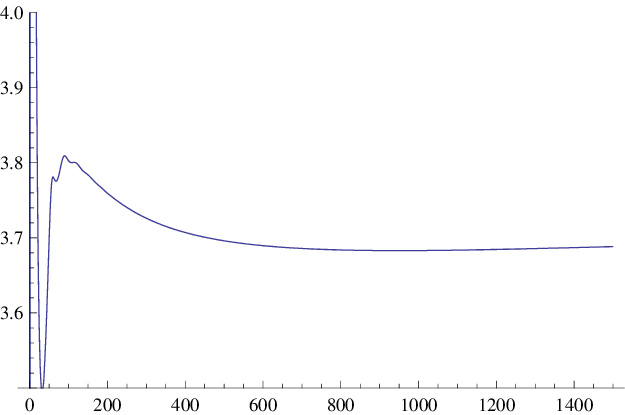}
\caption{The sequence $a_p$ for $f(x)=13x^{12}$.}
\end{center}
\end{figure}

\begin{figure}[p]
\begin{center}

\includegraphics[width=0.75\textwidth]{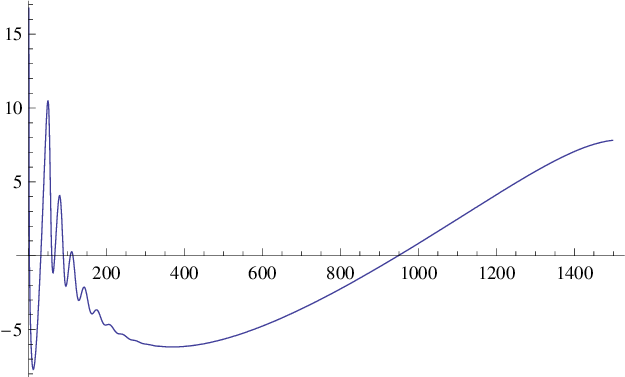}
\caption{The sequence $b_p$ for $f(x)=13x^{12}$.}
\end{center}
\end{figure}

Finally we look at $f(x)=13x^{12}$. It has $\Sigma'\approx\{0.105896\pm 1.97567i\}$, and our theorem does not apply in this situation. Nevertheless numerics show that $a_\infty$ exists, and its value is approximately $3.688371$ (see Figure 5). The sequence $b_p$ in Figure 6 doesn't seem to follow the asymptotic prescribed by $\sigma_1$. It is unclear how to pick $\delta$ for such a function since the error term $\eps_p^{(12)}$ does not have good decay when $\delta<0$. Apparently Assumption 3 is not necessary for the Main Theorem to hold, but in this case the structure of solutions to \eqref{eq:finalrecurrent} is unclear. 



\bibliographystyle{alpha}

\markboth{}{}

\bibliography{bibliography}

\end{document}